\documentclass[11pt]{article}   
\usepackage{amssymb,amscd,latexsym}   
\usepackage{amsmath}
\usepackage{amsthm}
\usepackage{mathdots}
\usepackage[pagebackref]{hyperref}
\usepackage{color}
\usepackage[all]{xy}
\newtheorem{Theorem}{Theorem}[section] \newtheorem{Corollary}[Theorem]{Corollary}  \newtheorem{Proposition}[Theorem]{Proposition} \theoremstyle{definition}
 \newtheorem{Definition}[Theorem]{Definition} \newtheorem{Example}[Theorem]{Example}  \newtheorem{Remark}[Theorem]{Remark}  \newtheorem{Question}[Theorem]{Question}   \newtheorem{conj}[Theorem]{Conjecture} \numberwithin{equation}{section}  
  
 \DeclareMathOperator{\beg}{indeg}

 \DeclareMathOperator{\Ext}{Ext}
 \DeclareMathOperator{\Tor}{Tor}
 \DeclareMathOperator{\Dim}{dim}
  \DeclareMathOperator{\pd}{pdim}
\DeclareMathOperator{\depth}{depth}
 \DeclareMathOperator{\Ht}{ht}
\DeclareMathOperator{\reg}{reg} 
\DeclareMathOperator{\grade}{grade}
\DeclareMathOperator{\End}{end} 
 
 \DeclareMathOperator{\E}{\textrm{E}}
 
 \DeclareMathOperator{\SD}{\textrm{SD}}
\DeclareMathOperator{\SDC}{\textrm{SDC}} 

\newcommand{\binomial}[2]{{#1 \choose #2}}

\def\rank{{\rm rank}\,}

\newcommand{\fm}{\mathfrak{m}}

\newcommand{\pp}{\mathbb{P}}

\newcommand{\ra}{\rightarrow}

\newcommand{\ud}{\underline d}
\newcommand{\od}{\overline d}

\def\phi{\varphi}

\def\aa{{\bf a}}
\def\bb{{\bf b}}
\def\dd{{\bf d}}
 \def\ff{{\bf f}}\def\aa{{\bf a}}

\newcommand{\lar}{\longrightarrow}\newcommand{\llar}{-\kern-5pt-\kern-5pt\longrightarrow}
\def\restr{{\kern-1pt\restriction\kern-1pt}}
\newcommand{\sslash}{\mathbin{/\mkern-6mu/}}

\begin{document}
\begin{center}
{\Large{\bf\sc Bounds for the degree and Betti sequences along a graded resolution}}
	\footnotetext{AMS Mathematics
		Subject Classification (2010   Revision). Primary 13A02,  13C05, 13C14, 13D02, 13D05; Secondary  14J17, 14J70.} 

	\vspace{0.3in}
	
	{\large\sc W.  A. da Silva}\footnote{Under a PhD fellowship from CAPES (88882.3311072019-01))} \quad
         	{\large\sc S. H. Hassanzadeh}\footnote{Partially supported by the MathAmSud project ``ALGEO''}
         	\quad
	{\large\sc A.  Simis}\footnote{Partially
		supported by a CNPq grant (301131/2019-8).}

\end{center}



\begin{abstract} 
	The main goal of this paper is to size up the minimal graded free resolution of a homogeneous ideal in terms of its generating degrees. By and large, this is too ambitious an objective. As understood, sizing up means looking closely at  the two available parameters: the shifts and the Betti numbers. Since, in general, bounds for the shifts can behave quite steeply, we filter the difficulty by  the subadditivity of the syzygies. The method we applied is hopefully new and sheds additional light on the structure of the minimal free resolution. For the Betti numbers, we apply the Boij-S\"oderberg techniques in order to get polynomial upper bounds for them. 
	It is expected that the landscape of hypersurface singularities already contains most of the difficult corners of the arbitrary case. In this regard, we treat some facets of this case, including an improvement of one of the  basic results of a recent work by Bus\'e--Dimca--Schenck--Sticlaru.
\end{abstract}

\section{Introduction}
Let $S=k[x_0,\ldots,x_n]$ be a standard graded polynomial ring over a field $k$ and let $I=(f_1,\cdots,f_r)\subset S$ be a homogeneous  ideal.
Our main purpose is to understand the numerical details of the graded minimal free resolution of $R:=S/I$ over $S$, by which one means the sequence of  degrees (shifts) and the sequence of the  Betti numbers.
Naturally, other numerical invariants of $I$ are important, such as, for example, the Castelnuovo--Mumford regularity, the socle degree of its Artninan reduction, the saturation exponent and the reduction number. 
Of course, in this generality no major results are expected, thus driving us to certain restrictions, e.g., on the projective (homological) dimension of $S/I$ over $S$.

By and large, the paper draws upon two major tools: spectral sequences and the Boij--S\"oderberg theory of Betti diagrams. Each of these is applied in a different direction, to be detailed in a minute.
Now, as usual, there is a motivation behind the scenes that explains how one has come thus far. In this work we have been initially moved by 
 the overture results of \cite{HS} regarding the homology of  ideals of $k[x,y,z]$ generated by three homogeneous forms of the same degree.
As it turns, some of it can actually be carried on to arbitrary number of variables as long as the homological  dimension is $3$.
Thus, for example, \cite[Lemma 1.1 and Theorem 1.5]{HS} can thus be refreshed:
\begin{Theorem}\label{THS} Let $S=k[x_0,\ldots,x_n]$ be a standard graded polynomial ring over a field $k$ and let
$I\subset S$ be a $3$-generated homogeneous $d$-equigenerated ideal  with  minimal graded free resolution
\begin{equation}\label{eqresolution}
0\rightarrow\bigoplus_{m=1}^{r-2}S(-D_m)\xrightarrow{\phi_3}\bigoplus_{i=1}^{r}S(-d_i)
\xrightarrow{\phi_2}\bigoplus^{3}S(-d)\xrightarrow{\phi_1}S\rightarrow S/I \rightarrow 0,
\end{equation}
where $D_1\geq\cdots\geq D_{r-2}$ and $d_1\geq\cdots\geq d_{r}$.
Then 
\begin{itemize}
\item[\rm (i)]{$D_m\geq d_m +1$ for all $1\leq m \leq r-2$.}
\item[\rm (ii)]{$\reg(R/I)\leq 3d-3$. }
\item[\rm (iii)]{$r \leq 3d-2$.}
\end{itemize}
\end{Theorem}
One reason it made sense to revisit some facets of the above result is that they have grabbed the attention of other authors, mainly in the case when $I$ is the Jacobian ideal  of a form $f\in k[x,y,z]$, where the theorem applies to explain the numerical invariants attached  to the singularity of the plane curve $V(f)$ (see, e.g., \cite{CDI19},\cite{DS19},\cite{DS20}). These recent achievements motivated further work encompassing  hypersurface singularity in higher projective space (see \cite{BDSS}). 

We next briefly describe the contents of the paper. Throughout, $S$ denotes a standard graded $k$-algebra and $R=S/I$, with $I$ homogeneous.

Section $2$ is devoted to subadditivity estimates for the degrees of  a resolution. 
It is divided in two subsections. In the first of these the gist of the typical assertion, as compared, e.g., to the subadditivity results in \cite{ACI}, lies in two directions: first, we assume that the standard graded $k$-algebra $S$ is a Koszul algebra (not just a polynomial ring); second, the subadditivity estimates involve  the degrees of both the minimal free $S$-resolution of $R$ and the minimal free $R$-resolution of $k$.
Note that neither of the two free resolutions is finite in general.
The basic subadditivity result is Proposition~\ref{Ptibi}, where the results depend on a certain intertwining of the degrees from the two free resolutions. 
The main tool employed in the proof is the change of ring spectral sequence
$$\Tor_r^S(k,R)\otimes_k\Tor_s^R(k,k)\Rightarrow \Tor_{r+s}^S(k,k).$$
We draw some corollaries, first regarding estimates of the degrees of the $S$-resolution of $R$ in a so-called `linear slope" case; second, regularity intertwining  estimates; third, estimates for the Green--Lazarsfeld invariant in certain condition. The contents are two technical to describe here, so we refer directly to the stated results in the subsection.

Now, for the second subsection. Here, we assume that $S$ is a polynomial ring and that the minimal free $S$-resolution of $R$  has a structure of $DG$-algebra.
The main result is Theorem~\ref{PDGres}. As a sample, one shows that a certain intertwining hypothesis implies certain subadditivity for $R$ over $S$. Precisely, setting $t_i^S(N):=\sup\{j\in \mathbb{Z} | (\Tor_i^S(k,N))_j\neq 0\},$ it is shown that
$$t_{i+1}^R(k)< t_i^S(R) \Rightarrow t_i^S(R) \leq \max\{t_j^S(R)+t_{i-j}^S(R)| 1\leq j\leq i-1\}.$$
We end this subsection with a few examples regarding the existence of a $DG$-algebra structure.

In section $3$ we assume again that $S$ is a standard graded polynomial ring and deal with a more direct estimate of the degrees and Betti numbers of the minimal free $S$-resolution of $R=S/I$, where $I$ is  homogeneous and besides $d$-equigenerated.
One main tool here is the Boij--S\"oderberg theory of Betti diagrams.
Our first concern is to bound the first Betti number of $I$, which is its minimal number of generators under the present hypothesis.
Though a well-know upper bound is known in terms of the generating degree $d$ and the projective dimension of $S/I$ over $S$, no efficient lower bound seems to be exhibited earlier.
The first result of the section gives a lower bound for $\beta_1(S/I)$ in terms of the upper degree sequence of the minimal free $S$-resolution of $S/I$ and $\Ht I$ (Proposition~\ref{Pmingensbound}). We believe that this lower bound in the non-pure case is new even in the case where $S/I$ is Cohen--Macaulay.
In addition, in the case the free resolution is $d$-linear it implies a binomial coefficient kind as lower bound. There is also a lower bound in terms of the Green--Lazarsfeld $N_q$ condition.

In addition, in the case of projective dimension $4$, assuming quadratic upper bounds for the upper degree terms of the resolution we deduce cubic upper bounds for the corresponding Betti numbers.
The expressions involved are too technical to reproduce here, so we refer to the details in the appropriate proposition (Proposition~\ref{PBettisbound}).

Section  $4$ deals with the case where $I$ is the Jacobian ideal of a form $f\in S=k[x_1,\ldots, x_n]$ in the spirit that hypersurface singularities may encompass most of the bad corners of the theory in the arbitrary case of equigenerated ideals.
We first elaborate on an example of a form $f$ of arbitrary degree $\geq 3$ in four variables such that the  upper degree sequence of the corresponding minimal resolution fails to be increasing.
In this example the failure to be increasing happens at the tail of the resolution. A question arises as to whether this failure may happen somewhere else along the resolution and whether the examples belong to a certain class of forms.

As a second issue, we deal with an upper bound for the regularity of $J_f^{\rm sat}/J_f$, where $J_f$ is the Jacobian ideal of $f$.
In fact, we prove a far out bound for $I/J_f$, where $I\subset S$ is a homogenous ideal such that $J_f\subset I\subset J_f^{\rm sat}$ (Proposition~\ref{SD_all}).
Our main hypothesis is a generalized sliding depth condition on the ideal $I$ as regards its minimal number of generators as compared to $\dim S=n$.
This result gives a strong version  of  the recent \cite[Theorem 3.1]{BDSS}, which assumes that $I$ is a perfect ideal of codimension $2$. Since ours holds under the assumption that $I$ is strongly Cohen--Macaulay, the consequential relation is obvious.
Perhaps more important is that our result also makes sense when  the singularity of $V(f)\subset \pp^{n-1}$ has codimension larger than $2$, a case which is not dealt with in \cite[Theorem 3.1]{BDSS}.

\medskip

As a final note we mention a connection with the recent \cite{DE} where the authors pose questions on the Betti numbers of certain monomial ideals satisfying the $N_{d,q}$ property, Definition \ref{DGL}. Our results in Section 3 (Corollaries \ref{C31},\ref{CC2}, \ref{CCT}  and  \ref{upperN})) provide answers to  some of these questions and not just for monomial ideals. For example, Corollary \ref{C31} explains why ideals with $N_{d,q}$ must have many generators.

\section{Subadditivity bounds via change of ring}


 Let $S$ denote a  Noetherian graded ring  over a field $k$ and  let $N$ stand for a finitely generated graded $S$-module. 
For any integer $i\geq 0$, set 
$$t_i^S(N):=\End(\Tor_i^S(k,N))=\sup\{j\in \mathbb{Z} | (\Tor_i^S(k,N))_j\neq 0\}.$$
Given an integer $n$, the $n$th  Castelnuovo--Mumford regularity of $N$ over $S$ is
$\reg_n^S(N):=\max\{t_i^S(N)-i \,|\, i\leq n\}$.  
The graded ring $S$ is said to be a Koszul algebra (over $k$) if $\reg^S(k):=\max\{\reg_n^S(k)\}_n=0$. Examples of Koszul algebras abound and include graded polynomial rings (see \cite{Con} for an account).  

We denote by ${\rm pdim}_S(N)$ the projective (i.e., homological) dimension of $N$ over $S$, a possibly infinite invariant.

\begin{Definition}\label{DGL}
Given integers $d, q\geq 1$, we say  that a homogeneous ideal $I\subset S$ satisfies the condition $N_{d,q}$ if $t^S_i(S/I)=d+i-1$ for all $1\leq i\leq q$.  Thus $N_{d,1}$ is the condition that $I$ is generated in the single degree $d$, $N_{d,2}$ adds the condition that $I$ is linearly presented. The Green-Lazarsfeld condition $N_q$ is, in this notation, $N_{2,q}$.  
\end{Definition}

The homological tool used in the following results are  based on spectral sequence techniques, for which we refer to \cite{Ch19}, \cite{Eis} and \cite{R} for any unexplained notation and properties.

\subsection{The base ring is a Koszul algebra}

Throughout this section we assume that $S$ is a standard graded Koszul algebra over the field $k$ and $R=S/I$ is a homogeneous residue algebra.
In the sequel we will consider the above invariants both over $S$ and $R$.

Even if $S$ is a polynomial ring, the departing generating degrees may have little impact on later syzygy degrees, as is the case in well-known examples in the theory.
For example, it has been shown in \cite[Theorem 6.2]{CC19} that, for any  positive integer $s\geq 9$, there exists a non-degenerate homogeneous  prime ideal $P$ of a polynomial ring over $k$ with maximal generating degree $6$ and a minimal syzygy of degree $s$ 
(see also \cite{U} for the source of these and other `bizarre' examples).

In this section, we investigate various aspects of the subadditivity question, in particular the intertwining  along the change of rings from $S$ to $R$, both for the `degrees' as for  the regularity.

\begin{Proposition}\label{Ptibi}  Let $S$ be a standard graded Koszul algebra over  $k$, and let $R=S/I$ denote a residual graded algebra with $p:={\rm pdim}_S(R)$, possibly infinite. Then, for any $i\geq 0,$
\begin{enumerate}
	\item[{\rm (1)}]
$t_i^S(R)\leq \max\{t_{i-j}^S(R)+t_{j+1}^R(k)\,|\, j=1,\cdots,i\}.$
	\item[{\rm (2)}] $t_{i+1}^R(k)\leq \max\{t_{i-j}^S(R)+t_j^R(k)\,|\, j=\max\{0,i-p\},\cdots,i-1\},$
	with the convention that $\max\{0,i-p\}=0$ if $p$ is infinite.
\end{enumerate}
\end{Proposition}
\begin{proof} The proof is based on the change of ring spectral sequence
 $$\Tor_r^S(k,R)\otimes_k\Tor_s^R(k,k)\Rightarrow \Tor_{r+s}^S(k,k).$$
In order to study the maps in this spectral sequence, we introduce some basic intervening complexes.
Thus, let $K_{\bullet}^S$ and $F_{\bullet}$ denote the  minimal free resolution of $k$ over $S$ and   the minimal free resolution of $k$ over $R$, respectively. 
Since $S$ is a Koszul algebra, we can write
$$K_{\bullet}^S(k):\cdots\to S(-i)^{\beta^S_i(k)}\to S(-i+1)^{\beta^S_{i-1}(k)}\to\cdots\to S.$$
Set $K_i=S(-i)^{\beta^S_i(k)}$. 
Consider the bicomplex $(K_{\bullet}\otimes_S R)\otimes_RF_{\bullet}$ in the second quadrant
\begin{eqnarray}
\xymatrix{
                                           						&&&(K_{2}\otimes_S R)\otimes_RF_0\ar[d]\\
                                               &\cdots &(K_1\otimes_S R)\otimes_RF_1\ar[d]\ar[r]&(K_{1}\otimes_S R)\otimes_RF_0\ar[d]\\
 & (K_{0}\otimes_S R)\otimes_RF_2 \ar[r]& (K_0\otimes_S R)\otimes_RF_1\ar[r] &  (K_{0}\otimes_S R)\otimes_RF_0}
\end{eqnarray}
The horizontal spectral sequence collapses at the first step
\begin{equation}\label{Ehor1}^1\E_{hor}^{-j,i}=\left \{
\begin{array}{ccccccc}
0& \text{if~~} j\neq 0 \\
 k^{\beta^S_i(k)}(-i)& j=0.
\end{array}
\right.  \end{equation}

Note that the shifts in this convergence are due to the fact that $S$ is a Koszul algebra. 

The vertical spectral sequence has first terms $^1\E^{-i,j}_{ver}=\Tor^S_j(k,R)\otimes_RF_i$ with connecting homorphisms $\Tor^S_j(k,R)\otimes F_i\to \Tor^S_j(k,R)\otimes F_{i-1}$. The map $\phi_i:F_i\to F_{i-1}$ comes from the minimal  free resolution $F_{\bullet}$, hence  $\phi_i(F_i)\subset \fm F_{i-1}$, where $\fm$ is the irrelevant maximal ideal of $R$. On the other hand $\Tor^S_j(k,R)$ is annihilated by $\fm$. Then, the connecting homomorphism $\Tor^S_j(k,R)\otimes F_i\to \Tor^S_j(k,R)\otimes F_{i-1}$ is the zero map.
Thus $^2\E^{-i,j}_{ver}=\Tor^S_j(k,R)\otimes_RF_i$. By setting $T_i:=\Tor^S_i(k,R)$, we draw the second vertical spectral as follows:
\begin{eqnarray}
\xymatrix{
                                         		&&       \cdots&T_3\otimes F_1&T_3\otimes F_0\\
                                  \cdots&\cdots&T_2\otimes F_2\ar@{-->}[rru]^{d^2}&T_2\otimes F_1&T_2\otimes F_0\\
         \cdots  & T_1\otimes F_3\ar@{--}[rrruu]&T_1\otimes F_2&T_1\otimes F_1&T_1\otimes F_0\\
T_0\otimes F_4\ar@{..}[rrrruuu]&T_0\otimes F_3& T_0\otimes F_2& T_0\otimes F_1 &  T_0\otimes F_0}
\end{eqnarray}
Let $i\geq 0$ be an integer. Consider the $i$th diagonal in the above picture. Since $^1\E_{hor}$ collapses, $k^{\beta^S_i(k)}(-i)$ is the $i$th homology of the total complex. 
The convergence of the vertical spectral sequence implies that the $^{\infty}\E_{ver}$ terms on the $i$th diagonal filter  $k^{\beta^S_i(k)}(-i)$.

Now, let $a$ be an integer such that  $a>\max\{t_{i-j}^S(R)+t_{j+1}^R(k) | j=1,\cdots,i\}$. 
Since $F_{\bullet}$ is minimal, if $t_{i-j}^S(R)\geq i-j$ and $t_{j+1}^R(k)\geq j+1$ then $a> i+1$. Therefore, considering these spectral sequences in degree $a$, one has$(^1\E_{hor}^{0,i})_a=0$, and hence, $(^{\infty}\E_{ver}^{-j,i-j})_a=0$ for all $j$. 

We now show that $(^{\infty}\E_{ver}^{0,i})_a=(T_i\otimes F_0)_{a}=\Tor^S_{i}(k,R)_a$. To see this, consider the map 
\begin{eqnarray*}
\xymatrix{
(T_{i-1}\otimes F_2)_a\ar[rr]^{d^2}&& (T_i\otimes F_0)_a.
}
\end{eqnarray*}
Since $a>\End(T_{i-1})+t_2^R(k)=\End(T_{i-1}\otimes F_2)$ then $(T_{i-1}\otimes F_2)_a=0$. Thus, $(^{3}\E_{ver}^{0,i})_a=(^{2}\E_{ver}^{0,i})_a=(T_i\otimes F_0)_{a}$. In the next page $(^{3}\E_{ver}^{0,i})_a$ is target of a map from $(^{3}\E_{ver}^{-3,i-2})_a$ which is a subquotient of $(T_{i-2}\otimes F_3)_a$. By a similar reasoning, the latter vanishes, hence eventually $(^{\infty}\E_{ver}^{0,i})_a=\Tor^S_{i}(k,R)_a$ which must be zero since the abutment is zero. Therefore $t_i^S(R)\leq \max\{t_{i-j}^S(R)+t_{j+1}^R(k): j=1,\cdots,i\}$, thus proving item (1).

To prove item (2), we regard the above spectral sequences  from a different angle.  The vertical spectral in the second step is the following 
\begin{eqnarray*}
\xymatrix{
                           &   \cdots&\Tor^S_2(k,R)\otimes F_1&\Tor^S_2(k,R)\otimes F_0\\
       &\Tor^S_1(k,R)\otimes F_2&\Tor^S_1(k,R)\otimes F_1&\Tor^S_1(k,R)\otimes F_0\\
\Tor^S_0(k,R)\otimes F_3\ar^{d^2}[rru]\ar@{-->}^{d^3}[rrruu]& \Tor^S_0(k,R)\otimes F_2&\Tor^S_0(k,R)\otimes F_1 &  \Tor^S_0(k,R)\otimes F_0}
\end{eqnarray*}
 Set $t_i:=t_i^S(R)$ and $\tau_i:=t_i^R(k)$.  We may assume that in the presentation $R=S/I$, the ideal $I$ has no linear form; so that   that $t_i\geq i+1$ for any $i\geq 1$.

 Notice that if $p$ is finite it is  the  last index $i$ for which $\Tor^S_i(k,R)\neq 0$. Let $i\geq 1$ and consider the $(i+1)$th diagonal in the above vertical spectral sequence  ($i=2$ is shown in the above picture). Let $a>\max\{t_{i-j}^S(R)+t_j^R(k):j=\max\{0,i-p\}\cdots,i-1\}$. Then $k(-i-1)_a=0$ since $a>i+1$. This implies that the infinity terms on the $(i+1)$th diagonal  in the above vertical spectral sequence  are all null at degree $a$.  Next, since $a>t_{i-j}+\tau_j$, $(T_{i-j}\otimes_RF_j)_a=0$. Therefore in any page, $n>2$,  $(^n\E^{-(i-j),j}_{ver})_a=0$. Hence any map with source in $(^2\E^{i+1,0}_{ver})_a$ maps to zero. This shows that $(^{\infty}\E^{i+1,0}_{ver})_a=(^{2}\E^{i+1,0}_{ver})_a=( \Tor^S_0(k,R)\otimes F_{i+1})_a=0$. The latter shows that $a>\tau_{i+1}$, since $( \Tor^S_0(k,R)\otimes F_{i+1})_{\tau_{i+1}}\neq 0$. Hence, $\tau_{i+1}\leq \max\{t_{i-j}+\tau_j:j=\max\{0,i-p\}\cdots,i-1\}$, as was to be shown.
\end{proof}

Next is a consequence regarding the special situation informally known as `linear slope'.
\begin{Corollary}\label{Clinearslope}
Let $S$ and $R$ be as in {\rm Proposition~\ref{Ptibi}}.  Suppose that $1\leq i\leq {\rm pdim}_S(R)$ and  that for any $j<i$, $t_j^S(R)\neq t_{j+1}^R(k)$.  Then either $t_i^S(R)= t_{i+1}^R(k)$ or else the following inequalities hold:
\begin{itemize}
\item $t_i^S(R)< t_{i+1}^R(k)\leq (i-1)t_1^S(R)+1$, or 
\item $t_{i+1}^R(k)< t_{i}^S(R)\leq it_1^S(R)$.
\end{itemize}
\end{Corollary}
\begin{proof}
 The proof is by induction on $i$, for which consider the two sets of inequalities in Proposition \ref{Ptibi}. Setting $t_i:=t_i^S(R)$ and $\tau_i:=t_{i}^R(k)$, we have 
  \begin{equation}\label{Eti1}
 \tau_{i+1}\leq \max\{t_i,t_{i-1}+\tau_1,\cdots,t_1+\tau_{i-1}\}.
  \end{equation}
   \begin{equation}\label{Eti2}
 t_i\leq\max\{\tau_{i+1},t_1+\tau_i,\cdots,t_{i-1}+\tau_2\}
 \end{equation}
Supposing $t_i<\tau_{i+1}$ the first assertion follows from (\ref{Eti1}) by the inductive hypothesis. Similarly, if $t_i>\tau_{i+1}$ then the second assertion follows from (\ref{Eti2}) by the inductive hypothesis.
\end{proof}

\begin{Corollary}\label{C of tis} Let $S$ and $R$ be as in {\rm Proposition~\ref{Ptibi}}. Then, for any $i\geq 0,$
\begin{enumerate}\itemsep2pt
\item[{\rm (1)}]{$t_i^S(R)\leq 2i+\sum_{j=2}^{i+1}\reg_j^R(k)$. }\label{P2}
\item[{\rm (2)}]{$\reg_i^S(R)\leq \reg_{i-1}^S(R)+\reg_{i+1}^R(k)+1$.}\label{P2'}
\item[{\rm (3)}]{$\reg_{i+1}^R(k)\leq \reg_{i}^S(R)+\reg_{i-1}^R(k)-1$ for $i\le p;$ and \\
	$\reg_{i+1}^R(k)\leq \reg_{p}^S(R)+\reg_{i-1}^R(k)-1$ for $i\ge p$, provided $p<\infty$.}\label{P2''}
\item[{\rm (4)}]{\rm {(}\cite[Proposition 5.8]{HI}) If $\reg^S(R)=1$, then $R$ is a Koszul algebra.}\label{P2'''}
\end{enumerate}
\end{Corollary}
\begin{proof} Item (1) follows from Proposition~\ref{Ptibi}(1) by induction on $i$. (2) is merely the definition of the regularity as applied in Proposition~\ref{Ptibi}(1), noting that $\reg_{j+1}^R(k)\geq  \reg_{j}^R(k)$.  Similarly, (3) follows from Proposition~ \ref{Ptibi}(2).  To see  (4), notice that (3) implies that  $\reg^R(k)\leq \max\{\reg_1^R(k),\reg^R_0(k)\}=0$.
 \end{proof}

\begin{Remark}\rm
If  $S$ is a  standard graded polynomial ring over $k$, then Corollary~  \ref{C of tis}(4) may be argued as follows: since $\reg^S(R)=1$, considering the reverse lexicographic order on $S$,  a theorem of Bayer and Stillman shows that $\reg^S(S/({\rm in}(I)))=1$. Therefore,  $I$ has a quadratic Groebner basis, hence $S/I$ is a Koszul algebra.
\end{Remark}

\begin{Corollary}\label{C of tis1} Let $S$ and $R$ be as in {\rm Proposition~\ref{Ptibi}}.
Suppose that $\reg_{n+1}^R(k)=0$ for some $n\geq 1$. Then, for every $i\leq n,$
\begin{enumerate}\itemsep2pt
  \item[{\rm (1)}]{$t_i^S(R)\leq 2i $.}\label{P3} 
\item[{\rm (2)}]{$\reg_i^S(R)\leq \reg_{i-1}^S(R)+1$.}\label{P7}
\item[{\rm (3)}]{  $t_i^S(R)\leq t_{i-1}^S(R)+2$ for $i\leq \min\{n,\depth(S)-\Dim(R)\}$.}\label{P5}
\end{enumerate}
\end{Corollary}
\begin{proof} (1) It is an immediate consequence of Corollary \ref{C of tis}(1). 
	
	(2) It  follows from Corollary~\ref{C of tis} (2).
	
	 (3) Set $m(R)\colon=\min\{i\geq 0| t_i^S(R)\geq t_{i+1}^S(R)\}$. By the proof of  \cite[Lemma 6.1]{ACI} we have $\Ext^{m(R)}_S(R,S)\neq 0$ hence $m(R)\geq \depth(S)-\Dim(R)$. The result now follows from     Proposition \ref{Ptibi}(1) as  $t_{j+1}^R(k)=j+1$ for $j\leq n$.
	
\end{proof}
Recall that, given an integer $q\geq 0$,  $R$ satisfies the Green--Lazarsfeld condition $N_q$ over $S$ if 
 $t_i^S(R)=i+1$ for $1\leq i \leq q$; or equivalently, if $\reg_q^S(R)=1$.
\begin{Corollary}\label{CNq} Let $S$ and $R$ be as in {\rm Proposition~\ref{Ptibi}}.
Suppose that $\reg_{n+1}^R(k)\leq 1$ for some $n\geq 1$ and that $R$ satisfies the Green-Lazarsfeld condition $N_q$ over $S$ for some $q\geq1$. Then,  $t_i^S(R)\leq 2i-q+1$ for $q+1\leq i\leq n$.
\end{Corollary}
\begin{proof}
Setting $t_i:=t_i^S(R)$ and $\tau_i:=t_{i}^R(k)$.

First, we show that $\reg_q^S(R)\leq1$ implies that $\reg_{q+1}^R(k)=0.$ This is done by induction on $i$, we will show that $\tau_i=i$ for $1\leq i\leq q+1$.

For $i=1$, $\tau_1=1$.  Assume that $i\geq 2$ and that the result holds for $i<q+1$. By Proposition \ref{Ptibi}(2)
$$\tau_{i+1}\leq \max\{t_{i-j}+\tau_j\,|\, j=\max\{0,i-p\},\cdots,i-1\}.$$ Note that $i-j\leq i\leq q$. Since $R$ satisfies the Green-Lazarsfeld condition $N_q$, $t_{i-j}=i-j+1$. By the inductive hypothesis $\tau_j=j$, hence $\tau_{i+1}\leq i+1.$

Now we prove that
$t_i^S(R)\leq 2i-q+1$ for $q\leq i\leq n$. Proceed by induction on $i$. The result holds for $i=q$ by  hypothesis. Assume that $i\geq q+1$ and that the result holds for $i$ and below. By Proposition \ref{Ptibi}(1) $$t_{i+1}\leq \max\{t_{i+1-j}+\tau_{j+1}\,|\, j=1,\cdots,i+1\} $$

We analyze the terms $t_{i+1-j}+\tau_{j+1}$ for $j=1\dots q $. By the initial part of proof, $\tau_{j+1}=j+1$. On the other hand, using the hypothesis or induction hypothesis, $t_{i+1-j}\leq i+2-j$ or $t_{i+1-j}\leq2(i+1-j)-q+1$. In both cases, $t_{i+1-j}+\tau_{j+1}\leq 2(i+1)-q+1$. 

Now we analyze the terms $t_{i+1-j}+\tau_{j+1}$ for $j=q+1,\dots i$. By hypothesis
$\tau_{j+1}\leq j+2$.  On the other hand, using the hypothesis or induction hypothesis, $t_{i+1-j}\leq i+2-j$ or $t_{i+1-j}\leq2(i+1-j)-q+1$. In both cases, $t_{i+1-j}+\tau_{j+1}\leq 2(i+1)-q+1$. Finally, the case that $j=i+1$, $t_{i+1-j}+\tau_{j+1}=\tau_{j+2}\leq i+3\leq2(i+1)-q+1.$ Thus concluding the proof.
\end{proof}
\begin{Remark}\label{Rti<2i} 
(a) Granted the assumption $\reg_{n+1}^R(k)=0$, the inequality in Corollary~\ref{C of tis1}(1) has been proved earlier in  \cite[Corollary 5.2]{ACI} in the case where $S$ is a polynomial ring. The argument, there,  uses the structure of minimal model, a tool that may not be available for a Koszul algebra.  
Moreover, as we see from the proof above, $t_i^S(R)\leq 2i $  is an immediate consequence of Proposition~ \ref{Ptibi}(1) which holds without any assumption on the regularity of $k$.

(b) Corollary \ref{C of tis1}(2) shows how the jumps happen along the way to compute the regularity of Koszul algebras. This inequality has also been shown in \cite[Proposition 6.7]{ACI} in the case where $S$ is a polynomial ring and $k$ has characteristic zero or prime characteristic $p$  such that $p\nmid {n\choose j}$ for any $j\le n$.

(c) Corollary~\ref{CNq} is a characteristic-free result. It improves the bound in  \cite[Theorem 7.1]{ACI} in the case where $k$ is a field, $q\geq 3$ and char$(k)$ fails for the above restriction.
\end{Remark}


\subsection{A role of the $DG$-algebra structure}


Throughout this part,  $S=k[X_1, \cdots, X_n]$ is a standard graded polynomial ring over $k$ and $I\subset S$ is a homogeneous ideal containing no linear forms. Write $R:=S/I=k[x_1,\cdots,x_n]$.

Let $K_{\bullet}:=K_{\bullet}(x_1,\cdots,x_n,R)$ stand for the Koszul complex of $\{x_1,\ldots,x_n\}\subset R$. 
For simplicity, set  $H_i:=H_i(x_1,\cdots,x_n;R)$, the $i$th homology $R$-module of  $K_{\bullet}$.
As is well-known, $H_{\bullet}:=\oplus_{i=0}^pH_i$ has a natural structure of DG-algebra, where $p=n-\grade (x_1,\cdots,x_n)$, and  $H_{i,j}:=(H_i)_j\simeq \Tor_i^S(k,R)_j$. 
We adhere to the usual notation, induced by the algebra structure, by writing a product $H_{i,j}H_{u.v}$ to mean the $k$-vector space spanned by the mutual product of respective elements.
With this notation 
$$H_0\langle H_{i-j}H_j| 1\leq j\leq i-1\rangle$$ 
denotes is the set of elements in $i$th koszul homology module $H_i$ which are  generated by product of elements of degrees strictly less than $i$.

Clearly $t_i^S(R)\geq i+1$ for $i>0$. If  $\reg_{m+1}^R(k)=0$ for some integer $m$,  Corollary~\ref{C of tis1} yields  $t_i^S(R)\leq 2i $ for $i\leq m$.  
It is shown in \cite[Theorem 5.1]{ACI} that  $H_{i,2i}=(H_{1,2})^i$ for $i\leq m$ and, by a similar argument, it is shown in \cite{BDGMS} that $H_{i,2i-1}=(H_{1,2})^{i-2}H_{2,3}$. It looks quite relevant to understand the next strand $H_{i,2i-2}$ given that, for example,  the linear strand $\langle H_{i,i+1}\rangle _i$ may not generate the whole algebra $H_{\bullet}$ even for monomial ideals, as  is shown in Example~\ref{ExIy} below.

The next theorem is the main result of the section. It gives a subadditivity estimate for $t_i^S(R)$ via obtaining a sufficient condition in order that  higher Koszul homology strands be related to  lower ones.
\begin{Theorem}\label{PDGres}  Suppose that  the graded minimal free resolution of $R=S/I$ over $S$ admits a structure of an associative DG-algebra.  Given  $i\geq 2$, one has:
\begin{enumerate}
	\item[\rm (i)] If $\Tor_{i+1}^R(k,k)_{\tau}=0$, for some integer $\tau$, then
$$H_{i,\tau}\subseteq   H_0\langle H_{i-j}H_j| 1\leq j\leq i-1\rangle_{\tau}.$$
In particular, this inclusion holds whenever $\tau>t_{i+1}^R(k)$ or $\tau<s_{i+1}^R(k):=\beg(\Tor_{i+1}^R(k,k))$.
	\item[\rm (ii)]  Suppose that 	$\Ht(I)\geq 2$. If  $t_{i+1}^R(k)< t_i^S(R) $ then 
	$$t_i^S(R) \leq \max\{t_j^S(R)+t_{i-j}^S(R)| 1\leq j\leq i-1\}.$$
\end{enumerate}
\end{Theorem}

\begin{proof} (i) Since $K_{\bullet}(x_1,\cdots,x_n,R)=R\otimes_S  K_{\bullet}(x_1,\cdots,x_n,S)$ and $K_{\bullet}(x_1,\cdots,x_n,S)$ is the minimal free resolution of $k$ as an $S$-module, then  $H_{\bullet}\simeq \Tor^S_{\bullet}(R,k)$ as $R$-algebras.  Considering $k$ as a module over this DG-algebra, one has $\Tor_1^H(k,k)=H_+/H_+^2=H_1$ (see \cite[Construction 2.3.2 and Proposition 3.2.4]{Av1} and the references thereof). In particular,  for any integer $a\geq 2$, $\Tor_1^H(k,k)_a=H_a/\langle H_lH_{a-l}\rangle _{1\leq l \leq a-1}$ where  $\langle H_lH_{a-l}\rangle _{1\leq l \leq a-1}$  is spanned by sums of product of elements of degrees strictly less than $a$. In other words, $(\Tor_1^H(k,k)_a)_b$ is the space of `fresh' generators in $H_{a,b}$ of degree $b$, in the sense that they are not obtainable  from homology in lower degrees.

By  \cite[Theorem 3.2]{Av}, there exists a spectral sequence 
 $$^2\E_{p,q}=(\Tor_{p}^{\Tor^S(R,k)}(k,k))_q\Rightarrow\Tor_{p+q}^R(k,k)\sslash f_{\ast} $$ 
 where $ f_{\ast}:\Tor^S(k,k)\to \Tor^R(k,k)$ is the natural map and $\Tor_{p+q}^R(k,k)\sslash f_{\ast}$ is a quotient of $\Tor_{p+q}^R(k,k)$.
The quoted theorem gives that  $^2\E_{p,q}= ~^{\infty}\E_{p,q}$, since the minimal free resolution of $R$ over $S$ admits a DG-algebra structure. Now, choose  $\tau$ such that  $\Tor_{i+1}^R(k,k)_{\tau}=0$. Then the convergence of the above spectral sequence implies that all of the infinity terms in the  $i$th diagonal are zero in degrees $\tau$. In particular $((\Tor_{1}^{\Tor^S(R,k)}(k,k))_{i-1})_{\tau}=0$. This vanishing amounts to saying that $H_{i-1,\tau}$ has no fresh generators.

(ii) By part (i), for any $\tau>t_{i+1}^R(k)$, $H_{i,\tau}\subseteq  H_0\langle H_{i-j}H_j| 1\leq j\leq i-1\rangle_{\tau}$. Since $\End(H_{i-j}H_j)\leq \End(H_{i-j})+ \End(H_{j})=t_{i-j}^S(R)+t_j^S(R)$, $H_{i,\tau}$ vanishes for $\tau>\max\{t_j^S(R)+t_{i-j}^S(R)| 1\leq j\leq i-1\}$, hence  $t_i^S(R) \leq \max\{t_j^S(R)+t_{i-j}^S(R)| 1\leq j\leq i-1\}$.  
\end{proof}

The next result provides additional insight into the  obstruction  for a graded free resolutions to admit a structure of a DG-algebra. 
\begin{Corollary}\label{CHi,i+1} Under the assumption of {\rm Theorem~\ref{PDGres}}, if $\reg_{m+1}^R(k)=0$ for some $m$, then  $H_{i,\tau}\subseteq  H_0\langle H_{j,j+1} | 1\leq j\leq i\rangle_{\tau}$ for any $i\leq m$ and for any $\tau$.
\end{Corollary}
\begin{proof}  Since $\reg_{m+1}^R(k)=0$,  $t^R_{i+1}(k)=i+1$ for any $i\leq n$. The result now follows from Theorem~\ref{PDGres} (i).
\end{proof}


Some examples of classes of  ideals in a Noetherian local (or standard graded ring) whose  free resolution admits a DG-algebra structure are as follows. 
 \begin{Example} (i) Any ideal $I\subset S$ having a free resolution of length $\leq 3$  (\cite[Proposition 1.3]{BE}).
 	
 	(ii)  Any Gorenstein ideal having free resolution of length $4$ (\cite{K87} , \cite{KM}). 
 	
 	(iii)  (char$(k)\neq 2$) $I\subset S$  is a Cohen-Macaulay almost complete intersection having free resolution of length $4$ (\cite{K94}). 
 	
 	(iv) If $I$ is a complete intersection  then for any integer $m$, a free resolution of $S/I^m$ (\cite{Sh}).
 \end{Example}
 

 We now visit a couple of known negative examples, viewed as an application of the previous propositions.

\begin{Example}\label{ExIy} (Asked by Sather-Wagstaff) Does the graded minimal  free resolution of the edge ideal of a plane graph admit the structure of a DG-algebra structure?

The answer is negative, the following example being atributted to Eisenbud.  Let $I=(X_1X_2,X_2X_3,\cdots,X_6X_7,X_7X_1)\in k[X_1,\cdots,X_7]$.
Its Betti table is
$$
\begin{tabular}{c|clclcl}

$\beta^S_{\bullet}(R)$&0&1&2&3&4&5\\
\hline
0&1&--&--&--&--&--\\
1&--&7&7&--&--&--\\
2&--&--&7&14&7&1
\end{tabular}.
$$
 Since  $I$ is generated by quadrics,  $R=S/I$ is a Koszul algebra and thus $\reg_R(k)=0$, but $H_{5,7}\not \subseteq  H_0\langle H_{j,j+1} | 1\leq j\leq 4\rangle_{7}$ because $H_{3,4}=H_{4,5}=0$ ,while $H_{5,7}\neq 0$.

Therefore, Corollary~\ref{CHi,i+1} implies that the minimal free resolution of $R$ over $S$ does not  admit any (associative) DG-algebra structure. 
It is also clear from the Betti table that this ideal satisfies $t_i^S(R)\leq i(t_1^S(R))=2i$ for all $i$.
\end{Example}

\begin{Example}\label{ExCavig} G. Caviglia, \cite{Ca}, introduced a family of three-generated ideals whose regularity grows quadratically in terms of  the degrees of the generators. Let $I_d=(X_1^d,X_2^d,X_1X_3^{d-1} -X_2X_4^{d-1})\in k[X_1,X_2,X_3,X_4]$ then $t_2^S(R)=d^2$. If  if $d>2$, $t_2^S(R)> 2d$. Therefore, Theorem~\ref{PDGres} implies that the minimal free resolution of $R=S/I$ does not admit a structure of a DG-algebra. 

If $d=2$, then  we have the following Betti tables 
\begin{center}
\begin{tabular}{c|clclcl}

$\beta^S_{\bullet}(R)$&0&1&2&3&4\\
\hline
0&1&--&--&--&--\\
1&--&3&--&--&--\\
2&--&--&5&4&1
\end{tabular}
\quad and \quad
\begin{tabular}{c|clclcl}

$\beta^R_{\bullet}(k)$&0&1&2&3&4&5\\
\hline
0&1&4&9&16&25&36\\
1&--&--&--&2&12&42
\end{tabular}
\end{center}
Here, though the criterion of Theorem~\ref{PDGres} (i) does not apply, this example has $H_{2,4}\neq H_{1,2}^2$. Indeed,  since $H_1=(H_{1,2})$ and $H_2=(H_{2,4})$, if $H_2=H_1^2$ then $I$ would be a complete intersection by a theorem of Tate and Assmus \cite[2.3.11]{BH}, which is not the case.  Also by degree reasoning $H_{3,5}, H_{4,6}\not \in H_0\langle H_{1,2}\rangle$. Notice that   $\reg^R_3(k)=1\neq 0$, hence this example shows that the condition $\reg_{n+1}^R(k)=0$, assumed for the proof of the equality $H_{i,2i}=(H_{1,2})^i$ in \cite[Theorem 4.1]{ACI}, is necessary..
\end{Example}


\section{Degree sequences along a graded resolution}

We keep the proviso that $S=k[x_1,\ldots,x_n]$ is a standard graded polynomial ring over a field $k$.


Let $M$ be a finitely generated graded $S$-module. The minimal free resolution of $M$ over $S$ has the following shape  
\begin{equation}\label{Eminres}
0\to \bigoplus_j S(-j)^{\beta_{p,j}}\to\cdots\to \bigoplus_j S(-j)^{\beta_{1j}}\to \bigoplus_j S(-j)^{\beta_{0,j}}\to M\to 0,
\end{equation}
where $\pd(M)=p$.
\begin{Definition}\label{Dbettiseq1}  $\beta_{i,j}$ is  the $(i,j)$th {\em graded Betti number}; $\beta_i=\sum_j \beta_{i,j}$ is the $i$th {\em Betti number}; $(\beta_0\,\ldots\,\beta_p)$ is called the {\em Betti sequence} of $M$.   

Set $\bar{d}_i :=t^S_i(M)= \max\{j|\beta_{i,j}(M)\neq 0\}$,  $\underline{d}_i := \min\{j|\beta_{i,j}(M)\neq 0\}$.
Throughout we focus on the two integer sequences $\bar{\dd}(M):=(\bar{d}_0,\ldots,\bar{d}_p)$ and  $\underline{\dd}(M):=(\underline{d}_0,\ldots,\underline{d}_p)$, referred to as the {\em upper degree sequence} and the {\em lower degree sequence} of $M$, respectively.
\end{Definition}

The following result is well-known. We give a proof for the reader's convenience.

\begin{Proposition}\label{Pascending} With the above notation, let $M$ be a finitely generated graded  $S$-module. Then:
\begin{enumerate}
	\item[{\rm (i)}] $\underline{\dd}(M)$ is strictly increasing.
	\item[{\rm (ii)}]  If $M$ is  Cohen--Macaulay then $\bar{\dd}(M)$ is also strictly increasing. 
\end{enumerate}  
\end{Proposition}
\begin{proof} (i) Given $0\leq v\leq p$, let 
$$f_v:\bigoplus_j S(-j)^{\beta_{v,j}}\to  \bigoplus_j S(-j)^{\beta_{(v-1),j}}$$ 
stand for the differential in the complex (\ref{Eminres}). Denote the basis of the free module $S(-j)^{\beta_{v,j}}$ by $\{e_i\}$ and that of $S(-j)^{\beta_{(v-1),j}}$ by $\{t_i\}$. Let $\deg(e_h)=\underline{d}_v$ and $\deg(t_r)=\underline{d}_{v-1}$.
Set  
$f_v(e_h)=\sum_ia_it_i$. 
Since the resolution is minimal, there are no null columns in the presentation matrix of $f_v$. Say, $a_i\neq 0$, for some $i$. Then
$$\underline{d}_v=\deg(e_h)=\deg(f_v(e_h))\geq  \deg(a_i)+ \deg(t_i)\geq 1+ \deg(t_r)=1+\underline{d}_{v-1}.$$

(ii) Since $M$ is Cohen-Macaulay, $\Ext_S^j(M,S)=0$ if and only if $j\neq n-p$. Dualizing the minimal resolution \ref{Eminres} into $S$, we obtain a  minimal free resolution of $\Ext_S^{e-p}(M,S)$. Moreover, $\bar{\dd}(M)=\underline{\dd}(\Ext_S^{n-p}(M,S)).$
Thus, the result follows from part (i).
\end{proof}

   The example in Proposition~\ref{PbadBetti} shows that the Cohen-Macaulay hypothesis is necessary in order to $\bar{\dd}(M)$ be strictly increasing.
 
 We will draw upon a fundamental result of the Boij--S\"oderberg theory.
 For an excellent introduction to the basic objects and notation we refer to \cite{F}).
 We recall a few of these.
 
Fix an integer $t\leq n$.
A sequence $\dd = (d_0,...,d_t) \in \mathbb{Z}^{t+1}$
 is a {\em degree sequence} (of length $t + 1$) if $d_{i-1}< d_i$ for $i = 1,...,t$. 
 
 Let $\mathbb{Z}^{t+1}_{\rm deg}$ denote
 the set of degree sequences of length $t+ 1$. Given two degree sequences $\dd$ and $\dd'$ in  
 $\mathbb{Z}^{t+1}_{\rm deg}$, we say $\dd\preccurlyeq \dd'$ if $d_i \leq d'_i$ for $i=0,...,t$. 
 
 For $\aa,\bb\in \mathbb{Z}^{t+1}_{\rm deg}$ with $\aa \leq \bb$,
 we introduce the `window' $\mathcal{D}(\aa,\bb):=\{\dd\in\mathbb{Z}^{t+1}_{\rm deg}| \aa\leq \dd\leq \bb\}$. 
 If $\dd=(d_0,\ldots ,d_p)\in \mathbb{Z}^{p+1}_{\rm deg}$  and $t\leq p$, then we set $\tau_t(\dd) = (d_0,\ldots,d_t)$.

 A finitely generated graded $S$-module $M$ is called pure of type $\dd = (d_0,\ldots,d_p)$ provided  $\beta_{i,j}(M)\neq 0$ if
 and only if $j = d_i$ for $i = 0,\ldots,p$. Hence, a pure module has a graded minimal free resolution of the form
 $$
 0\to S(-d_p)^{\beta_{p,d_p}}\to\cdots\to S(-d_1)^{\beta_{1,d_1}}\to S(-d_0)^{\beta_{0,d_0}}\to M\to 0.
 $$
 Clearly, in this case, $\beta_i(M)=\beta_{i,d_i}(M)$ for every $0\leq i\leq p$.
 Now, according to \cite[Theorem 1]{HeKu}, the case where a pure module $M$ is Cohen--Macaulay is characterized by certain values of the corresponding Betti numbers, namely,
 \begin{equation*}\label{EHerzog}
 	\beta_i(M)=:\beta_{i,d_i}(M)=\prod_{1\leq j\leq p ,\, j\neq i}\frac{|d_j-d_0|}{|d_j-d_i|}.
 \end{equation*}
 (Note a tiny discrepancy with the notation in \cite{HeKu}, where $d_0=0$).
 Each degree sequence $\dd = (d_0,\ldots, d_p)$ then defines a ray in the cone of Betti diagrams and there is a unique point $\pi(\dd)$ on this ray such that, for the corresponding module,  one has $\beta_{0,d_0}=1$, i.e., the module is cyclic.  One writes  $\beta_{0,d_0}(\pi(\dd))=1$ and, more generally, $\beta_{i,d_i}(\pi(\dd))$ for the Betti number of the corresponding module. With this notation, one can write
 \begin{equation}\label{EHerzog}
 	\beta_i(\dd)=:\beta_{i,d_i}(\pi(\dd))=\prod_{1\leq j\leq p,\,  j\neq i}\frac{|d_j-d_0|}{|d_j-d_i|}.
 \end{equation}
 \begin{Theorem}\label{TESBS}{\rm (\cite[Theorem 0.2]{ES},\cite{BoSo}\rm )} Let $M$ be a graded $S$-module of projective dimension $p$ and codimension $c$. Then the Betti diagram $\beta(M)$ can be decomposed as a sum
 	$$\beta(M)=\sum_{c\leq t\leq p}\sum_{\dd\in \mathcal{S}_t}q_{\dd}(\beta(\pi(\dd))),$$
 \end{Theorem}
 \noindent where $q_{\dd}$ are positive rational numbers whose sum is $1$ and  for every $t$, $\mathcal{S}_t$ is a saturated chain of $t$-uples between $\tau_t(\underline{\dd}(M))$ and  $\tau_t(\bar{\dd}(M)))$ in the sense of the ordering $\preccurlyeq$ introduced earlier.

 \subsection{Bounding $\beta_1$ and $\beta_c$}
 
 By assumption, $\beta_1(S/I)=\dim_k [I]_d=\mu(I)$, the minimal number of generators of $I$.
The latter has an obvious upper bound in terms of $d$ and $p=\pd (S/I)$. To see this,  we may assume that $k$ is an infinite field. Since  $\pd(S/I)=p$, $\depth(S/I)=n-p$. Thus, one can specialize modulo  a linear  sequence of length $n-p$ which is regular both in $S$ and on $S/I$. Letting $\bar{S}$   denote the residue of $S$  modulo this regular sequence, one has $\Tor_i^{\bar S}(S/I,k)=\Tor_i^{S}(S/I,k)$ for all $i$ (\cite[p 140, Lemma 2]{M}, also \cite[Proposition 1.1.5]{BH}). Thus, to compute the Betti numbers of $S/I$ we may assume that $p=n$. Therefore,
$$\dim_k [I]_d\leq\dim_k(S_d)=\binomial{d+p-1}{p-1}.$$

By drawing upon Theorem~\ref{TESBS}, in the next proposition we recover this bound, but also establish a non-trivial lower bound for $\beta_1$ and $\beta_c$, where $c=\Ht(I).$


\begin{Proposition}\label{Pmingensbound} Let $S=k[x_1,\ldots,x_n]$ be a standard graded  polynomial ring over the field $k$ and let $I\subset S$ be a homogeneous ideal generated in degree $d$ and of height $c\geq 2$.
With the notation of the previous subsection, one has
\begin{enumerate}

\item[{\rm (1)}]$$ \frac{d\underline{d}_2\cdots\underline{d}_{c-1}}{(\overline{d}_c-d)(\overline{d}_c-\underline{d}_2)\ldots(\overline{d}_c-\underline{d}_{c-1})}\leq \beta_c(S/I).$$

\item[{\rm (2)}] $$\frac{\od_2\ldots\od_c}{(\od_2-d)\ldots(\od_c-d)}\leq \mu(I)\leq \binomial{d+p-1}{p-1}$$
\end{enumerate}
\end{Proposition}
\begin{proof} 

 Theorem ~\ref{TESBS} says that
\begin{equation}\label{eqgeral}
    \beta(S/I)=\sum_{c\leq s\leq p} \text{ \ }\sum_{\bf d\in \mathcal{D}(\tau_s(\underline{\bf d}(S/I)),\tau_s(\overline{\bf d}(S/I)))}q_{\bf d}\beta(\pi(\bf d)),
\end{equation}

An element of $\mathcal{D}(\tau_s(\underline{\bf d}(S/I)),\tau_s(\overline{\bf d}(S/I)))$ is of the form $\mathbf{d}_{i_2,\ldots, i_s}=(0,d,\ud_2+i_2,\cdots,\ud_s+i_s)$. Here for any 
$2\leq j\leq s-1,$
\begin{equation}\label{E1}
\ud_j+i_j+1\leq \ud_{j+1}+i_{j+1},{\rm and}
\end{equation}
\begin{equation}\label{E2}
0\leq i_j\leq \od_j-\ud_j.
\end{equation}

Let $b_{i_2,\ldots,i_s}^{\{j\}}$ denote the nonzero entry on the $j$th column of $\beta(\pi(\mathbf{d}_{i_2,\ldots, i_s}))$. That is,
\begin{equation}\label{geral}
b_{i_2,\ldots,i_s}^{\{j\}}=\frac{d(\ud_2+i_2)\ldots \widehat{ (\ud_j+i_j)}\ldots(\ud_s+i_s)}{(\ud_j+i_j-d)\ldots(\ud_j+i_j-\ud_{j-1}-i_{j-1})(\ud_{j+1}+i_{j+1}-\ud_j-i_j)(\ud_s+i_s-\ud_j-i_j)}.
\end{equation}

$(1).$ According to the decomposition \ref{eqgeral}
\begin{equation}\label{essa}
    \beta_c(S/I)=\sum_{c\leq s\leq p}\sum_{(i_2,...,i_s)}q_{i_2...i_s} \cdot b_{i_2,\ldots,i_s}^{\{c\}}.
\end{equation}

Note that by definition $b_{i_2,\ldots,i_s}^{\{c\}}\geq b_{i_2,\ldots,i_c}^{\{c\}}$ for all $s\geq c$. Consequently, 
\begin{equation}
    \beta_c(S/I)\geq\sum_{t\leq s\leq p}\sum_{(i_2,...,i_s)}q_{i_2...i_s} \cdot b_{i_2,\ldots,i_c}^{\{c\}}.
\end{equation}
 Theorem \ref{TESBS} also says that $\sum_{t\leq s\leq p}\sum_{(i_2,...,i_s)}q_{i_2...i_s}=1$. Therefore, a lower bound for $b_{i_2,\ldots,i_c}^{\{c\}}$ which is independent of $(i_2,...,i_c)$ provides a lower bound for $\beta_c(S/I)$.

Recall that $$b_{i_2,\ldots,i_c}^{\{c\}}=\frac{d(\ud_2+i_2)\ldots (\ud_{c-1}+i_{c-1})}{(\ud_c+i_t-d)(\ud_c+i_c-\ud_2-i_2)\ldots(\ud_c+i_c-\ud_{c-1}-i_{c-1})}.$$

Obviously,
 $$\frac{\ud_j+i_j}{\ud_c+i_c-\ud_j-i_j}\geq\frac{\ud_j}{\ud_c+i_c-\ud_j}\geq \frac{\ud_j}{\od_c-\ud_j}\quad  \text{\rm for}\quad  2\leq j\leq c-1.$$ Hence, 
$$ \beta_c(S/I) \geq\frac{d\underline{d}_2\cdots\underline{d}_{c-1}}{(\overline{d}_c-d)(\overline{d}_c-\underline{d}_2)\ldots(\overline{d}_c-\underline{d}_{c-1})}.$$

$(2).$  We first consider the case where $S/I$ is Cohen-Macaulay. 
Then  $c=p$, hence the formula of Theorem~\ref{TESBS} becomes
\begin{equation}\label{Edec1}
\beta(S/I)=\sum_{\bf d\in \mathcal{D}((\underline{\bf d}(S/I)),(\overline{\bf d}(S/I)))}q_{\dd}(\beta(\pi(\dd))).
\end{equation}


\begin{equation}\label{Eb1}
b_{i_2,\ldots,i_p}^{\{1\}}=\frac{(\ud_2+i_2)\ldots(\ud_p+i_p)}{(\ud_2+i_2-d)\ldots(\ud_p+i_p-d)},
\end{equation}
According to the decomposition (\ref{Edec1}), 
$$\beta_1(S/I)=\sum_{(i_2,...,i_p)}q_{i_2...i_p} \cdot b_{i_2,\ldots,i_p}^{\{1\}}.$$
Since $\sum_{(i_2,...,i_p)}q_{i_2...i_p}=1$, an upper bound (respectively, a lower bound) for $b_{i_2,\ldots,i_p}^{\{1\}}$ which is independent of $(i_2,...,i_p)$ provides an upper bound (respectively, a lower bound) for $\beta_1(S/I)$.

Now, we can think of $b_{i_2,\ldots,i_p}^{\{1\}}$ as a positive real function.  In terms of any of the variables $i_2,\cdots,i_p$, it is a hyperbolic function with negative vertical asymptotic; thus, the maximum value of  $b_{i_2,\ldots,i_p}^{\{1\}}$ is attained at the minimum values of $i_j$'s and the minimum values are attained at the maximum values of $i_j$'s. We then have
$$\frac{\od_2\ldots\od_p}{(\od_2-d)\ldots(\od_p-d)}\leq~~  b_{i_2,\ldots,i_p}^{\{1\}}\leq \frac{\ud_2\ldots\ud_p}{(\ud_2-d)\ldots(\ud_p-d)}.$$
The function in the right hand  side is  hyperbolic  in terms of any among $\ud_2,\cdots,\ud_p$ in the domain  $[d+1,\infty)$. Thus, the maximum values are attained at the minimum values of each variable, so one gets
$$\frac{\od_2\ldots\od_p}{(\od_2-d)\ldots(\od_p-d)}\leq~~  b_{i_2,\ldots,i_p}^{\{1\}}\leq \frac{(d+1)\ldots(d+p-1)}{(p-1)!}=\binomial{d+p-1}{p-1}.$$
Consequently, the decomposition (\ref{Edec1}) yields
$$\frac{\od_2\ldots\od_p}{(\od_2-d)\ldots(\od_p-d)}\leq \beta_1(S/I)\leq \binomial{d+p-1}{p-1}.$$

Now, assume the general case, where $c\leq p$.  Then, according to (\ref{Edec1}), 
$$\beta_1(S/I)=\sum_{c\leq s\leq p}\sum_{(i_2,...,i_s)}q_{i_2...i_s} \cdot b_{i_2,\ldots,i_s}^{\{1\}}.$$
According to  (\ref{Eb1}), a similar argument as above shows that
$$\frac{\od_2\ldots\od_s}{(\od_2-d)\ldots(\od_t-d)}\leq~~  b_{i_2,\ldots,i_s}^{\{1\}}\leq \frac{\ud_2\ldots\ud_s}{(\ud_2-d)\ldots(\ud_s-d)}.$$
Obviously, 
$$\frac{\ud_2\ldots\ud_s}{(\ud_2-d)\ldots(\ud_s-d)}\leq \frac{\ud_2\ldots\ud_p}{(\ud_2-d)\ldots(\ud_p-d)} \quad  \text{\rm for}\; s\leq p$$
and 
$$\frac{\od_2\ldots\od_s}{(\od_2-d)\ldots(\od_s-d)}\geq \frac{\od_2\ldots\od_c}{(\od_2-d)\ldots(\od_c-d)} \quad  \text{\rm for}\quad  s\geq c.$$

Finally, since $\sum_{c\leq s\leq p}\sum_{(i_2,...,i_s)}q_{i_2...i_s}=1$, we get
$$ \frac{\od_2\ldots\od_c}{(\od_2-d)\ldots(\od_c-d)}\leq \sum_{c\leq s\leq p}\sum_{(i_2,...,i_s)}q_{i_2...i_s} \cdot b_{i_2,\ldots,i_s}^{\{1\}} \leq \frac{\ud_2\ldots\ud_p}{(\ud_2-d)\ldots(\ud_p-d)}.$$
This yields the assertion.

\end{proof}

Finding good lower bounds  for Betti  numbers has been dealt with by several authors. Perhaps the leading question in this direction is the so-called  Eisenbud-Horroks conjecture. The following result of Hezog and Kuhl is well-known:
 \begin{Theorem}\label{Herzog}{\rm (\cite{HeKu}) } If $M$ is a graded S-module of projective dimension
$p$ with a linear resolution, then $\beta_i(M)\geq \binomial{p}{i}$.

 \end{Theorem}
The lower bounds in Proposition~\ref{Pmingensbound} have the following consequence:
 \begin{Corollary}\label{C31}
 Let $S$ be a standard graded polynomial ring over a field $k$ and let $I\subset S$ be a homogeneous ideal of height $c\geq 2$ and projective dimension $p$ with a $d$-linear free resolution. Then
 \begin{enumerate}

\item[{\rm (1)}]
$$\max\left\{ \binomial{p}{c},\binomial{d+c-2}{c-1}\right\}\leq \beta_c(S/I);$$

\item[{\rm (2)}]$$\max\left\{p, \binomial{d+c-1}{c-1}\right\}\leq \mu(I)
.$$
 \end{enumerate}
 More generally, if $I$ satisfies the condition  $N_{d,q}$  {\rm (cf. \ref{DGL})}, then $\mu(I)\geq \binomial{d+c'-1}{c'-1}$ where $c'=\min\{c,q\}$.
 \end{Corollary}
 \begin{proof} 
 By Theorem \ref{Herzog} we already know that $\beta_i\geq\binomial{p}{i}$. For the case that $I$ has a linear free resolution, one uses  Proposition \ref{Pmingensbound} with $\ud_j=\od_j=d+j-1$, for $2\leq j\leq p$. The proof for the number of generators  is similar.
 \end{proof}

When the ideal $I$ satisfies the condition $N_{d,c}$, we still have a lower bound for $\beta_c$.

 \begin{Corollary}\label{CC2}
 
 Let $S$ be a standard graded polynomial ring over a field $k$ and let $I\subset S$ be a homogeneous ideal of height $c\geq 2$. If I satisfies condition $N_{d,c}$ then $\binomial{d+c-2}{c-1}\leq \beta_c(S/I).$
\end{Corollary}

The above Corollary  will be improved in the next proposition. We will give lower bounds not only for  $\beta_1$ and $\beta_c$ but for all Betti numbers. Furthermore, we will show that Betti numbers have polynomial upper bounds,  in terms of $p$ and $d$, and that a Betti number reaches such a bound if and only if all Betti numbers reach their bounds.

\begin{Proposition}\label{Minha}
Let $S$ be a standard graded polynomial ring over a field $k$ and let $I\subset S$ be a homogeneous ideal generated in degree $d$ of height $c\geq 2$ and projective dimension $p$ with linear free resolution. Then 
\begin{itemize}
\item[\rm (1)] for any $1\leq t\leq p$,
 $$\max\left\{\binomial{p}{t},\,\binomial{d+t-2}{t-1}\binomial{d+c-1}{c-t}\right\}\leq \beta_t(S/I)\leq \binomial{d+t-2}{t-1}\binomial{d+p-1}{p-t}=:C_t,$$ 

\item[\rm (2)] $\beta_t(S/I)=C_t$ for some $t$ if and only if $\beta_t(S/I)=C_t$ for all $t$;
\item[\rm (3)] {\rm (\cite{HeKu}, \cite[1.4.15]{BH})} If $S/I$ is a Cohen-Macaulay ring, then $\beta_t(S/I)=C_t$ for all $t$.

\end{itemize}

\end{Proposition}

\begin{proof} $(1).$  We keep the notation of Proposition \ref{Pmingensbound}. By Theorem \ref{Herzog} we already know that $\beta_i\geq\binomial{p}{i}$.
According to the decomposition \ref{eqgeral}

\begin{equation*}\label{essa}
    \beta_t(S/I)=\sum_{c\leq s\leq p}\sum_{(i_2,...,i_s)}q_{i_2...i_s} \cdot b_{i_2,\ldots,i_s}^{\{t\}}.
\end{equation*}

Note that, by definition,  $b_{i_2,\ldots,i_c}^{\{t\}}\leq b_{i_2,\ldots,i_s}^{\{t\}}\leq b_{i_2,\ldots,i_p}^{\{t\}}$ for all $c\leq s\leq p$ and that $b_{i_2,\ldots,i_c}^{\{t\}}=0$ whenever $t>c$. Therefore, the leftmost inequality below will be considered only for $t\leq c$.
\begin{equation}
    \sum_{t\leq s\leq p}\sum_{(i_2,...,i_s)}q_{i_2...i_s} \cdot b_{i_2,\ldots,i_c}^{\{t\}}\leq\beta_t(S/I)\leq\sum_{t\leq s\leq p}\sum_{(i_2,...,i_s)}q_{i_2...i_s} \cdot b_{i_2,\ldots,i_p}^{\{t\}}.
\end{equation}

To shorten the notation, when  $i_2=\cdots=i_s=0$, we denote $q_{i_2...i_s}$ by  $q_{s0}$. Similarly,   $b_{0,\ldots,0}^{\{t\}}$ with $s$ zeros in the index will be denoted  $b_{s0}^{\{t\}}$.
By hypothesis, $\ud_j=\od_j=d+j-1$, for $2\leq j\leq p$ and $i_2=i_3=\dots=i_p=0$. Hence,
$$\sum_{s=c}^{p}q_{s0} \cdot b_{c0}^{\{t\}}\leq \beta_t(S/I)\leq \sum_{s=c}^{p}q_{s0} \cdot b_{p0}^{\{t\}}, \; \; {\rm and} \;\; \sum_{s=c}^{p}q_{s0}=1.$$ 
Since
$$b_{p0}^{\{t\}}=\frac{d(d+1)\ldots \widehat{ (d+t-1)}\ldots(d+p-1)}{(t-1)!(p-t)!}=\binomial{d+t-2}{t-1}\binomial{d+p-1}{p-t}$$
and
$$b_{c0}^{\{t\}}=\frac{d(d+1)\ldots \widehat{ (d+t-1)}\ldots(d+c-1)}{(t-1)!(c-t)!}=\binomial{d+t-2}{t-1}\binomial{d+c-1}{c-t}$$
we have the desired inequalities.

$(2).$ Suppose that for some $t$,  $\beta_t(S/I)=C_t$. Then
\begin{equation}\label{b}
  \beta_t(S/I)=C_t=\sum_{s=c}^{p}q_{s0}b_{s0}^{\{t\}}=q_{c0}b_{c0}^{\{t\}}+\dots+q_{(p-1)0}b_{(p-1)0}^{\{t\}}+q_{p0}C_t.  
\end{equation}

Adding the fact that $\sum_{s=c}^{p}q_{s0}=1$, we conclude that

 $$q_{c0}(C_t-b_{c0}^{\{t\}})+\dots+q_{(p-1)0}(C_t-b_{(p-1)0}^{\{t\}})=0.$$
However, $C_t-b_{s0}^{\{t\}}>0$ and $q_{s0}\geq0$ for $c\leq s\leq p-1$.  Then $q_{c0}=\dots =q_{(p-1)0}=0$ and $q_{p0}=1$. Now, equation (\ref{b}) implies that $\beta_l(S/I)=C_l$ for all $l.$

$(3)$ Follows from  item $(2)$  with $c=p$.

\end{proof}
Note that the lower bound in Proposition \ref{Minha} is still valid if we only assume that the ideal $I$ satisfies the condition $N_{d,c}$, so we have:

\begin{Corollary}\label{CCT}
Let $S$ be a standard graded polynomial ring over a field $k$ and let $I\subset S$ be a homogeneous ideal generated in degree $d$ and of height $c\geq 2$. If $I$ satisfies condition $N_{d,c}$ {\rm (}i.e.,  the minimal graded resolution is linear up to step $c${\rm )},  then
$$\binomial{d+t-2}{t-1}\binomial{d+c-1}{c-t}\leq \beta_t(S/I)\,\,\, \text{for}\,\,\,\, 1\leq t\leq c.$$
\end{Corollary}

\subsection{General polynomial bounds for  Betti numbers}
We next establish  upper bounds for the Betti numbers in terms of polynomial functions in the defining degree $d$.
In projective dimension $3$ the bounds are best possible in general.
In projective dimension $4$ we get cubic bounds whenever the  highest degrees in a graded free resolution have certain quadratic upper bounds.

As before, let $(\beta_0,\beta_1,\cdots,\beta_p)$  denote the Betti sequence of $S/I$.
\begin{Proposition}\label{PBettisbound}
Let $S$ be a standard graded polynomial ring over the field $k$ and let $I\subset S$ be a $d$-equigenerated ideal of height $c\geq 2$ and projective dimension $p$.  Then:
\begin{enumerate}
\item[{\rm (1)}] If $p=3$ then $\beta_2\leq d(d+2) $ and $\beta_3\leq d(d+1)/2$.
\item[{\rm (2)}] If $p=4$, then
\begin{enumerate}
\item[{\rm (a)}] If $\od_2\leq d^2+4d+2$ then $\beta_2\leq d(d+2)(d+3)/2$, otherwise $$\beta_2\leq \frac{d(\od_2+1)(\od_2+2)}{2(\od_2-d)}\leq \frac{((3d-2)d^2+3)((3d-2)d^2+4)}{2((3d-2)d^2-d+2)}< (d+\frac{1}{8})((3d-2)d^2+2).$$
\item[{\rm (b)}]  If $\od_3\leq\max\{ d+2,(1/2)(d^2+2d-1)\}$ then $\beta_3\leq d(d+1)(d+3)/2;$ else,
$$\beta_3\leq \frac{d(\od_3-1)(\od_3+1)}{(\od_3-d)} \leq \frac{d((3d-2)d^2+2)((3d-2)d^2+4)}{((3d-2)d^2-d+3)}.$$
\item[{\rm (c)}]  If $\od_4\leq\max\{ d+3,(1/3)(d^2+2)\}$ then $\beta_4\leq d(d+1)(d+2)/6;$ else 
$$\beta_4\leq \frac{d(\od_4-2)(\od_4-1)}{2(\od_4-d)} \leq \frac{d((3d-2)d^2+2)((3d-2)d^2+3)}{((3d-2)d^2-d+4)}.$$
\end{enumerate}
\item[{\rm (3)}]  If $p\geq 5$ then for any $2\leq j\leq p $, 
$$\beta_j\leq d\max\left\{\frac{1}{j-1}\binomial{d+j-2}{j-2}\binomial{d+p-1}{p-j},\quad \frac{1}{(\od_j-d)}\binomial{\od_j+p-j}{p-j}\binomial{\od_j-1}{j-2} \right\}.$$
\end{enumerate}
\end{Proposition}
\begin{proof} We follow the same schedule of proof as in Proposition~ \ref{Pmingensbound}, where one could assume the Cohen--Macaulay case, as the general case worked  quite the same way.
Recall from this proof that $b_{i_2,\ldots,i_p}^{\{j\}}$ denotes the nonzero entry on the $j$th column of  the diagram $\beta(\pi(\mathbf{d}_{i_2,\ldots, i_p}))$.
	One has:
\begin{eqnarray*}
b_{i_2,\ldots,i_p}^{\{j\}}&=&\frac{d(\ud_2+i_2)\ldots \widehat{ (\ud_j+i_j)}\ldots(\ud_p+i_p)}{(\ud_j+i_j-d)\ldots(\ud_j+i_j-\ud_{j-1}-i_{j-1})(\ud_{j+1}+i_{j+1}-\ud_j-i_j)(\ud_p+i_p-\ud_j-i_j)}\\
&=&\frac{d}{\ud_j+i_j-d}\prod_{k=2}^{j-1}\frac{\ud_k+i_k}{\ud_j+i_j-\ud_k-i_k}\prod_{l=j+1}^{p}\frac{\ud_l+i_l}{\ud_l+i_l-\ud_j-i_j}.
\end{eqnarray*}

The relations  (\ref{E1}) and (\ref{E2}) imply the inequalities $$\frac{\ud_k+i_k}{\ud_j+i_j-\ud_k-i_k}\leq \frac{\ud_j+i_j-j+k}{j-k} \quad \text{\rm and}\quad  \frac{\ud_l+i_l}{\ud_l+i_l-\ud_j-i_j}\leq \frac{\ud_j+i_j+l-j}{l-j}.$$
One then gets $$b_{i_2,\ldots,i_p}^{\{j\}}\leq \frac{d}{\ud_j+i_j-d}\prod_{k=2}^{j-1}\frac{\ud_j+i_j-j+k}{j-k}\prod_{l=j+1}^{p} \frac{\ud_j+i_j+l-j}{l-j}.$$
We now inspect for which value of $\ud_j+i_j$ the right hand side of the above inequality attains its maximum value. Setting $x=\ud_j+i_j$, it becomes a  hyperbolic function  of $x=\ud_j+i_j$
\begin{equation}\label{Ef(x)}
f(x)=\frac{d}{(x-d)}\frac{(x-j+2)\ldots(x-1)}{(j-2)!}\frac{(x+1)\ldots(x+p-j)}{(p-j)!},
\end{equation}
which we wish to analyze in the range $[d+j-1,\od_j]$.

The behavior of $f(x)$ for $p=3$ and that for $p>3$ will be quite different. 

(1)  ($p=3$) If $j=2$, $f(x)=\frac{d(x+1)}{(x-d)}.$ The maximum value of this hyperbolic function occurs at the minimum value of $x$; so that $\beta_2\leq f(d+1)=d(d+2)$. When $j=3$, $f(x)=\frac{d(x-1)}{(x-d}$, similarly, $\beta_3\leq f(d+2)=d(d+1)/2$.

(2)  ($p=4$)  Say, $j=2$. $f(x)=\frac{d(x+1)(x+2)}{2(x-d)}$. We look  for $x\in [d+1,\od_2]$ wherein $f(x)=f(d+1)=d(d+2)(d+3)/2$. This amounts to find the roots of $(x+1)(x+2)=(x-d)(d+2)(d+3)$. One of the roots of this equation is $d+1$ hence the other root is $d^2+4d+2$. Therefore, if   $\od_2\leq d^2+4d+2$ the maximum value in the range $[d+1,\od_2]$ is $d(d+2)(d+3)/2$. Otherwise, the maximum value is $f(\od_2)$. To settle the last inequality, we appeal to the bound for the Castelnuovo-Mumford regularity ~\cite[Theorem 3.5(ii)]{bruns1998cohen}. Accordingly, $\reg(S/I)\leq (3d-2)d^2$ whenever $\Dim(S/I)\leq 2$. Then the last inequality follows by the fact that $\od_2-2\leq \reg(S/I)$. 

 The argument for $j=3,4$ is similar.

(3) ($p\geq 5$) Consider again the function $f(x)$ in (\ref{Ef(x)}). 
Since the numerator is a convex function in the range $x>j-2$, the intersection of $y=f(x)$ with any straight line in $\mathbb{A}^2$, in this range, consists of at most two points. Consequently, $f(x)$ has only one local minimum for $x>d+j-1$.
Therefore 
$$\max_{x\in [d+j-1,\od_j]}\{f(x)\}=\max\{f(d+j-1),f(\od_j)\}.$$
It is straightforward to see that $$f(d+j-1)=\frac{d}{j-1}\binomial{d+j-2}{j-2}\binomial{d+p-1}{p-j},\quad f(\od_j)= \frac{d}{(\od_j-d)}\binomial{\od_j-1}{j-2}\binomial{\od_j+p-j}{p-j}$$
\end{proof}

\begin{Corollary} \label{upperN}
Let $S$ be a standard graded polynomial ring over the field $k$ and let $I\subset S$ be  a homogenous ideal   of height $c\geq 2$ and projective dimension $p$ that satisfies the  condition $N_{d,q}$. Then $$ \beta_t\leq \binomial{d+t-2}{t-1}\binomial{d+p-1}{p-t} \, \,\text{for}\,\,\, 2\leq t\leq q.$$
\end{Corollary}
\begin{proof}
 Follows from Proposition \ref{PBettisbound}
\end{proof}

The results in  Proposition \ref{Minha}, Proposition \ref{PBettisbound} and Corollary \ref{upperN} encourage us to state the following
\begin{conj} Let $S$ be a standard graded polynomial ring over the field $k$ and let $I\subset S$ be  a homogenous ideal  generated in degree $d$ and of  projective  dimension $p$. Then all of the Betti numbers of  $S/I$ are bounded by polynomial functions of $d$ of degree at most $p-1$.
\end{conj}


\section{Hypersurface singularities}

Set $S=k[x_1,\cdots,x_n]$ and let $f\in S$ be a reduced form such that char$(k)$ does not divide $\deg f$.
We are interested in applying the preceding theory to the case where $R=S/J_f$, where $J_f\subset S$ denotes the gradient (or Jacobian) ideal of $f$. 
Hypersurface singularity theory is an important  field, hence the relevance of looking at $J_f$. The purpose of it in this landscape is the hope that it unravels most of the phenomena found in the case of arbitrary homogeneous equigenerated ideals.

It may help looking at some examples.

 \begin{Example}\rm
	This example  shows that, even if  the resolution admits   a DG-algebra structure, the inequality $t_i^S(R)\leq it_1^S(R)$ may fail to hold. 
	
	Let $J_f\subset k[x,y,z]$ denote the gradient ideal of the quartic $f=(x^2-z^2)(x^2-y^2)+y^4$. Set $R=S/J_f$ as usual.
\end{Example}
One has 
$$
\begin{tabular}{c|clclcl}
	
	$\beta^S_{\bullet}(R)$&0&1&2&3\\
	\hline
	0&1&--&--&--\\
	1&--&--&--&--\\
	2&--&3&--&--\\
	3&--&--&--&--\\
	4&--&--&3&--\\
	5&--&--&1&2
\end{tabular}
$$
Therefore, $t_2^S(R)=7>2t_1^S(R)$.

\begin{Example}\label{PbadBetti}\rm
Let $d \geq 2$ be an integer not diving char$(k)$ and let $f :=x^dy+y^dz+z^dw\in S=k[x,y,z,w]$, a polynomial ring in $4$ variables. Then one has:
\end{Example}
\begin{Proposition}
	Let $J :=J_f\subset S$ denote the gradient ideal of $f$ above. Then:
	\begin{enumerate}
		\item[\rm {(1)}] $V (f )\subset \mathbb{P}^3$ is an isolated singularity supported on the point $(0:0:0:1);$ in particular, $V (f )$ is reduced and irreducible and arithmetically normal.
		\item[\rm {(2)}]  $J$ is an ideal of linear type.
		\item[\rm {(3)}] Set $J^{\rm sat}:=J:\fm^{\infty}$, where $\fm=(x,y,z,w)$.
		Then $J^{\rm sat} = (J, x^{d-1}z^{d-1})$ and ${\rm Soc}(S/I)=J^{\rm sat}/J$. In particular, the socle degree of $R/J$ is $2d-2$, hence the highest degree in the tail of the graded minimal free resolution of $S/J$ is $2d-2+4=2d+2$.
		\item[\rm {(4)}]  The third term in the graded minimal free resolution of $S/J$ contains a summand in degree $3d-1;$ in particular,  $\bar{\dd}(S/J)$ is decreasing {\rm (}respectively, strictly decreasing{\rm )} at the tail provided $d\geq 3$ {\rm (}respectively, $d>3${\rm )}.
	\end{enumerate}
\end{Proposition}
\begin{proof}
	(1)  A calculation gives $J=(x^{d-1}y, x^d+dy^{d-1}z, y^d+dz^{d-1}w, z^d)$, from which is immediate that any prime ideal containing $J$ has to contain $(x,y,z)$. Therefore, the latter is the only minimal prime of $S/J$. The rest is clear: since $J$ has codimension $3$, then $S/(f)$ is normal, hence a domain as well since $f$ is homogeneous.
	
	(2)  This follows from a well-known criterion, namely, when $J$ is an almost complete intersection that is generically a complete intersection (\cite[Proposition 3.7]{SV1}), the latter being granted since $J$ admits the following linear syzygy $(x,-y, -dz, d^2w)^t$.
	
	(3) This is the gist of the example, specially if one does not wish to compute the entire resolution.
	We first observe that $x^{d-1}z^{d-1}\in J:\fm$. Thus, e.g., $x^dz^{d-1}= z^{d-1}(x^d+dy^{d-1}z) -dy^{d-1} z^d$, etc.
	Therefore, $G:=(J,x^{d-1}z^{d-1})\subset J\colon \fm\subset J^{\rm sat}$.
	
	{\sc Claim.} $G$ is a Gorenstein ideal.
	
	Indeed, up to signs and scalar multiples, the generators of $G$ are the maximal Pfaffians of the alternating matrix
	$$\mathcal{A}(G)=\left(
	\begin{array}{ccccc}
		0    & z^{d-1} &  0 &  -x  & y^{d-1}\\
		-z^{d-1} &   0   &  0  &  -y &  0\\
		0    &  0 &   0 &   dz & x^{d-1}\\
		x    &   y  & -dz  &  0  & -dw\\
		-y^{d-1} &  0 &  -x^{d-1} &  dw &  0
	\end{array}
	\right).
	$$
	From $J\subset G\subset J^{\rm sat}$ one derives $J:\fm^{\infty}\subset G:\fm^{\infty}\subset J^{\rm sat}$. Since $G$ is Gorenstein (unmixed would suffice) then $G=G:\fm^{\infty}$, hence $J^{\rm sat}=G$.
	
	The assertion about the socle and the socle degree is an immediate consequence.
	
	(4) Let $\partial f$ for short denote the partial derivatives of $f$ in the given order of the variables, where we have replaced $\partial f/\partial x$ by $(1/d)\partial f/\partial x$. Consider the vector $\ell=(x,-y, -dz, d^2w)$.
	
	{\sc Claim 1.}  Let $\mathcal{K}(\partial f)$ denote the Koszul matrix of $\partial f$. Then the syzygy matrix of $\partial f$ is the concatenated matrix
	$\left(\ell^t\, |\, \mathcal{K}(\partial f)\right).$ 
	
	{\sc Claim 2.} $(x^{d-1}z^{d-1}, \, -z^{d-1},\, 0,\, dy^{d-1},\, 0,\, 0,\, -dx^{d-1})^t$ is a minimal syzygy of $\left(\ell^t\, |\, \mathcal{K}(\partial f)\right).$
	
	The assertion in (4) is clearly an immediate consequence of these two claims since $\mathcal{K}(\partial f)$ is in degree $2d$, hence $2d+d-1=d+1+ 2(d-1)=3d-1$.
	In order to prove the two claims at once, we give the explicit graded minimal free resolution of  $J=(\partial f)$.
	Namely, we argue that the latter is
	$$0\to S(-(2d+2))\stackrel{\phi_4}\lar \begin{array}{c}S(-(2d+1))^4\\\oplus\\S(-(3d-1))\end{array}\stackrel{\phi_3}\lar \begin{array}{c}S(-(d+1))\\\oplus\\S(-2d)^6\end{array}\stackrel{\phi_2}\lar S(-d)^4\stackrel{\phi_1}{\lar} S,
	$$
	where  
	$$\phi_1=[\partial f], \quad \phi_2=\left[\ell^t\, |\, \mathcal{K}(\partial f)\right], \quad \phi_3=\left[\begin{array}{c|c}
		{\partial f} & x^{d-1}z^{d-1}\\
		\hline\\[-5pt]
		K(\underline{\ell})^t & \mathfrak{A}^t
	\end{array}\right], \quad
	\phi_4=\left[\begin{array}{c}
		\ell\,^t\\[5pt]
		\hline\\[-8pt]
		0
	\end{array}\right].	         $$
	Here  $K(\underline{\ell})^t$ denotes the transposed Koszul matrix of  the sequence $\underline{\ell}$ of the entries of $\ell$, while  $\mathfrak{A}=[-z^{d-1}, 0, dy^{d-1},0,0, -dx^{d-1}]$.
	The proof is based on the Buchsbaum--Eisenbud acyclicity criterion:
	
	$\bullet$ The maps define a complex.
	
	Straighforward calculation gives that $\ell^t$ is a syzygy of $\partial f$.
	Therefore,  relations $\phi_1\cdot \phi_2=0$ and $\phi_3\cdot \phi_4=0$ are clear. The relation $\phi_2\cdot \phi_3=0$ stems from an obvious duality. 
	
	$\bullet$ The ranks of the maps add up.
	
	Clearly, $\rank \phi_1=1$. Since $\phi_2$ is a submatrix of the first syzygies of $\partial f$, then $\rank \phi_2\leq 3$. On the other hand, $z^{3d}$ is an obvious nonzero minor of order $3$.
	Thus, the first two ranks add up.
	
	Similarly, to see that the ranks of $\phi_2$ and $\phi_3$ add up it suffices to see that $\phi_3$ has a nonzero minor of order $4$.
	There is no preferred such minor, but many  monomial such ones easy to pick.
	The rest is clear.
	
	$\bullet$ The Fitting ideals have enough depth.
	
	We need $\Ht I_3(\phi_2)\geq 2$. This is accomplished with the two $3$-minors $z^{3d}$ and $-x^{3(d-1)}y^3$.
	Next, we need $\Ht I_4(\phi_3)\geq 3$.  It will suffice to consider the minors of the last $6$ rows.
	Close inspection gives the following two $4$-minors up to a nonzero scalar multiple: $z^{d+2}$ and $x^{d-1}y^3$ corresponding to the submatrices
	$$\left[\begin{array}{cccc}
		-y  &-x & 0  &  -z^{d-1}\\
		dz  & 0   &  0 &  0\\
		0  & dz  &  0   &  0\\
		0  &  0  & -dz  & -dx^{d-1}
	\end{array}
	\right] 
	\quad \text{\rm and} \quad 
	\left[\begin{array}{cccc}
		-y  & 0  &  0  & -z^{d-1}\\
		0    & y  &    0  &  0\\
		0   & 0   &  y  &  0\\
		0   & -d^2w  & -dz & -dx^{d-1}	
	\end{array}
	\right] ,
	$$
	respectively.
	Therefore, any prime ideal $\wp$ containing $I_4(\phi_3)$ has to contain either $(x,z)$ or $(y,z)$. Another $4$-minor up to a nonzero scalar is $y^{d+2}+4y^2z^{d-1}w$. Thus, if $(x,z)\subset \wp$ then also $y\in\wp$, showing that $\wp$ has height at least $3$.
	The other case is analyzed in a similar way.
\end{proof}

\begin{Question}\rm
	Does the above example belong to a well structured class of forms?
\end{Question}

\begin{Remark} \rm
In the case of a hypersurface $V(f)\subset \pp^n$,  one can read out the bounds of $\beta_2$ in Proposition~\ref{PBettisbound}  as saying that the minimal number of generators of the module of vector fields leaving $V(f)\subset \pp^n$ invariant is of order at most $\mathcal{O}((d-1)^{p})$, where $p$ is the projective dimension of $S/J_f$ over $S$. Thus:

$\bullet$ In $\pp^2$, it is at most $d^2-1$. Note that in this particular case, Theorem~\ref{THS}(3) gives the upper bound  $3d-5$, which is much better.

$\bullet$ In $\pp^3$ and provided  $\reg(S/J_f)\leq (d^2-1)(d+2)/2-2$,  it is at most $(d^2-1)(d+2)/2$; our believe is that if $d$ is sufficiently large, then  $\reg(S/J_f)\leq (d^2-1)(d+2)/2-2$ is always the case.
\end{Remark}

Let us now look at bounds for the regularity of $S/J_f$.
In some special cases, the bound is linear in $d$, as proved in the recent \cite{BST} and \cite{BDSS}, via a previous result of \cite{RosTer}.
Namely, let $f=l_1l_2\cdots l_d$  denote the defining polynomial of a generic central arrangement in $S=k[x_1,\ldots,x_n]$ with $d\geq n\geq 2$ and char$(k)\nmid d$.
Then  both \cite[Theorem 2.4 (iii)]{BST} and  \cite[Proposition 2.8]{BDSS} give the exact value ${\rm reg}(S/J_f)=2d-(n+2)<n(d+2)$.
On the other hand, \cite[Theorem 3.1]{BDSS} provides a landscape where a form $f\in S$ is such that ${\rm reg}(S/J_f)< n(d+2)$.
This is then applied to the situation where $V (f)$ is a generic arrangement of surfaces with isolated singularities in $\mathbb{P}^3$, fulfilling certain transversality assumptions.
The gist of the main assumption in \cite[Theorem 3.1]{BDSS} is that the singularity locus of $V(f)\subset \pp^{n-1}$ has codimension two (in fact contains an arithmetically Cohen--Macaulay singularity of codimension two).

Our purpose here is to provide a far reaching generalization of \cite[Theorem 3.1]{BDSS} with no prior codimension restriction on the singularity locus.

Recall, for the purpose of notation, that given a set of elements $\ff=\{f_1,\ldots,f_r\}\subset S$, the cycle (respectively, homology) of order $i$ of the Koszul complex $K_{\bullet}(\ff;S)$ is denoted $Z_i(\ff;S)$ (respectively, $H_i(\ff;R)$).

\begin{Proposition}\label{SD_all}  Let $S=k[x_1,\cdots,x_n]$ denote a standard graded polynomial ring  over a field $K$.   Let $f\in S_d$ be a reduced form and let $I\subset S$ be a homogeneous ideal  such that $J_f\subseteq I\subseteq J_f^{\rm sat}$. Let  $I$ have height $h\geq 2$ and  minimally  generated by $\ff=\{f_1,...,f_r\}$. Assume  that 
\begin{enumerate}
\item[{\rm (i)}] Either $r\leq n-1$ and $\depth(H_i(\ff;R))\geq \min\{n-h,n+1-r+i\}$ for all non-zero homology$;$ or else 
\item[{\rm (ii)}] $r \geq n$, $I$ is saturated and satisfies $\depth(Z_i(\ff,R))\geq \min\{n+1-r+i, n+2-h, n\}$ for all $i\geq r-n+1$.
\end{enumerate}
Then $$\reg(S/J_f)\leq \max\{\reg(S/I), n(d-2)-(n-\Ht I)\beg(I/J_f)\}.$$
\end{Proposition}
{\em Proof.} By the standing assumption $\Ht(J_f:I)\geq n$, which gives that $(J_f:I)$ is an $n$-residual intersection of  $I$. We apply \cite[Proposition 3.2]{HN} with $k=1$, by which either of the conditions (i) and (ii) implies  $\SD_1$ or $\SDC_2$ in \cite[Proposition 3.2]{HN}. This implies  the inequality
	$$\reg(I/J_f)\leq  n(d-2)-(n-\Ht I)\beg(I/J_f),$$ 
since $I/J_f={\rm Sym}^1(I/J_f).$ 
Then the desired inequality follows from the behavior of the  Castelnuovo-Mumford regularity along the short exact sequence
$$\quad\quad\quad\quad\quad\quad\quad\quad \quad\quad 
0\ra I/J_f\ra S/J_f\ra S/I\ra 0. \quad\quad\quad\quad\quad\quad\quad\quad\quad\quad\quad\quad\quad\quad\square$$

Conditions $\SD_1$ and $\SDC_2$ are implied by condition SCM (Strongly Cohen-Macaulay), where all Koszul homology modules are Cohen--Macaulay.  The latter includes the class of ideals in the linkage class of complete intersections. Perfect ideals of height $2$ and perfect Gorenstein ideals of height $3$ are among scattered examples of this class.
This explains why  \cite[Theorem 3.1]{BDSS} is a consequence of  Proposition~\ref{SD_all}.

Note that the proposition applies to Example~\ref{PbadBetti}, by taking $I=J_f^{\rm sat}$.
Here one gets reg$(S/J_f)={\rm reg}(S/J^{sat})=3d-7$, 
while $n(d-2)-(n-\Ht J_f) {\rm indeg} (J^{sat}/J_f)= 2d-4$, which is the  standard degree at the tail of the  free resolution of  $S/J_f$.
\medskip

{\sc \bf Acknowledgment.} Part of this work started while the second author visited the University of Utah, in 2016, under a CNPQ (Brazil) fellowship.
He is indebted to the latter for the financial support. In addition, he
thanks S. Iyengar and A. Boocher for joyful discussions on the subject of the paper.


\noindent {\bf Addresses:}

\medskip

\noindent {\sc Wellington Alto da Silva}\\
Instituto de Matem\'{a}tica\\
Universidade Federal do Rio de Janeiro\\ 
21941-901 Rio de Janeiro, RJ,  Brazil\\
{\em e-mail} wellington10rj@ufrj.br

\medskip

\noindent {\sc Seyed Hamid Hassanzadeh}\\
 Instituto de Matem\'{a}tica\\
Universidade Federal do Rio de Janeiro\\ 
21941-901 Rio de Janeiro, RJ,  Brazil\\
{\em e-mail} hamid@im.ufrj.br 

\medskip

\noindent {\sc Aron Simis}\\
Departamento de Matem\'atica, CCEN\\ 
Universidade Federal de Pernambuco\\ 
50740-560 Recife, PE, Brazil\\
{\em e-mail}:  aron.simis@ufpe.br

\end{document}